\documentclass[10pt]{article}
\usepackage{amsmath}
\usepackage{amssymb}
\usepackage{stackrel,amssymb}
\usepackage{amsfonts}
\usepackage{amsthm}
\usepackage{slashed}
\usepackage{tikz-cd}
\usepackage{mathrsfs}
\usepackage[all]{xy}
\usepackage{mathtools}

\usepackage{mathabx,epsfig}

\usepackage{amssymb,amsfonts}
\usepackage{enumerate}
\usepackage{enumitem}
\usepackage{mathrsfs}
\usepackage{stmaryrd}
\usepackage{comment}

\usepackage{relsize}
\usepackage[bbgreekl]{mathbbol}
\usepackage{amsfonts}
\DeclareSymbolFontAlphabet{\mathbb}{AMSb} 
\DeclareSymbolFontAlphabet{\mathbbl}{bbold}

\newtheorem{thm}{Theorem}[subsection]

\newtheorem{lem}[thm]{Lemma}

\theoremstyle{definition}

\theoremstyle{remark}
\newtheorem{rem}[thm]{Remark}

\begin{document}

\newcommand\restr[2]{{
  \left.\kern-\nulldelimiterspace 
  #1 
  \vphantom{\big|} 
  \right|_{#2} 
  }}

\makeatletter
\renewcommand{\@seccntformat}[1]{%
  \ifcsname prefix@#1\endcsname
    \csname prefix@#1\endcsname
  \else
    \csname the#1\endcsname\quad
  \fi}
\makeatother

\makeatletter
\newcommand{\colim@}[2]{%
  \vtop{\m@th\ialign{##\cr
    \hfil$#1\operator@font colim$\hfil\cr
    \noalign{\nointerlineskip\kern1.5\ex@}#2\cr
    \noalign{\nointerlineskip\kern-\ex@}\cr}}%
}
\newcommand{\colim}{%
  \mathop{\mathpalette\colim@{\rightarrowfill@\textstyle}}\nmlimits@
}
\makeatother

\newcommand\rightthreearrow{%
        \mathrel{\vcenter{\mathsurround0pt
                \ialign{##\crcr
                        \noalign{\nointerlineskip}$\rightarrow$\crcr
                        \noalign{\nointerlineskip}$\rightarrow$\crcr
                        \noalign{\nointerlineskip}$\rightarrow$\crcr
                }%
        }}%
}

\title{A short geometric derivation of the dual Steenrod algebra}         
\author{Kiran Luecke}        
\date{\today}          
\maketitle

\begin{abstract}
    \noindent This two-page note gives a non-computational derivation of the dual Steenrod algebra as the automorphisms of the formal additive group. Instead of relying on computational tools like spectral sequences and Steenrod operations, the argument uses a few simple universal properties of certain cohomology theories.
\end{abstract}

All algebras are graded and commutative. $C_2$ is the cyclic group of order 2. 

\noindent
\textbf{Preamble:} Let $MT$ be the category of pairs $(F,s)$ where $F$ is a contravariant functor from the homotopy category of finite spectra to the category of graded abelian groups that takes sums to products and $s$ is a natural isomorphism from $F$ to $[1]\circ F\circ\Sigma$, where $[1]$ denotes the shift-of-grading functor. A morphism $(F,s)\to (F',s')$ is a natural transformation $F\to F'$ that takes $s$ to $s'$. Let $E$ be a homotopy ring spectrum and $E^*=\pi_{-*}E$. Consider the ``Conner-Floyd" functor
$\text{Alg}_{E^*}\rightarrow MT$
$$A\mapsto (X\mapsto EA^*X:=E^*X\otimes_{E^*}A).$$
Suppose that $E_*E$ is flat over $E^*$, so that $E(E_*E)$ is represented by $E\wedge E$. Consider the natural transformation
$$\eta_E:\text{Spec}E_*E\rightarrow \text{Hom}_{MT}(E,E-)$$
$$(E_*E\xrightarrow{f}A)\mapsto \eta_E(f):=(E^*X\xrightarrow{1\wedge\text{id}}(E\wedge E)^*X\simeq E(E_*E)^*X\xrightarrow{F}EA^*X)$$
The coproduct on $E_*E$ induced by the map $E\wedge E\xrightarrow{\text{id}\wedge 1\wedge\text{id}}E\wedge E\wedge E$ induces an ``internal" composition of morphisms $E\rightarrow EA$ in the image of $\eta_E$. If $EA$ is a cohomology theory (e.g. if $A$ is flat over $E^*$) then it corresponds to a (homotopy ring) spectrum which is an $E$-module, harmlessly denoted by $EA$, and there is an inverse to $\eta$ given by sending a transformation $T:E\rightarrow EA$ to the map on homotopy induced by
$E\wedge E\xrightarrow{\text{id}\wedge T}E\wedge EA\rightarrow EA.$

\begin{lem} Let $E$ be a real oriented homotopy ring spectrum such that $E_*E$ is flat over $E^*$. Write $\mathbb{G}_E$ for the formal group with coordinate $\text{Spf}E^*BC_2$ and write $\text{Aut}\mathbb{G}_E$ for the group-valued presheaf on $\text{Alg}_{E^*}$ sending $R$ to the group of strict automorphisms of $\mathbb{G}_H(R)$. Then there is a morphism $\text{Spec}(E_*E)\xrightarrow{ev_E}\text{Aut}\mathbb{G}_E$ of monoid-valued presheaves on $\text{Alg}_{E^*}$.

\end{lem}
\begin{proof} 
Let $e$ be the canonical generator of $E^*B C_2$. Let $R$ be an $E^*$-algebra and let $T$ be an element of $\text{Hom}_{MT}(E,ER)$. The element $T(e)$ may be identified with a power series $f(e)\in R[[x]]$. Moreover by multiplicativity of $T$ and naturality applied to the multiplication map $BC_2\times BC_2\rightarrow BC_2$, the power series $f(e)$ must be an endomorphism of the formal group law of $\mathbb{G}_E$. The linear coefficient $f_0$ is forced to be 1 by considering the pullback of $e$ along $S^1\hookrightarrow BC_2$ and invoking stability of $T$. Hence $f(e)$ defines a strict automorphism of $\mathbb{G}_E$. Precomposition of the assignment $T\mapsto f(e)$ with $\eta_E$ defines the desired map
$ev_E:\text{Spec} E_*E(R)\rightarrow\text{Aut}\mathbb{G}_E(R)$
which is clearly natural in $R$ and commutes with the monoidal structure on both sides by inspection.
\end{proof}

\begin{lem} Suppose that there is a homotopy commutative ring spectrum $M$ such that 
\begin{enumerate}
    \item There is a map $M\rightarrow H$ inducing an isomorphism on $\pi_0$.
    \item $M$ is real oriented with real orientation $e_M\in M^1BC_2$ and formal group $\mathbb{G}_M$ with formal group law $F_M$.
    \item $M_*M$ is flat over $M^*$, and the map $ev_M(R)$ afforded by the previous lemma is injective when $MR$ is a cohomology theory.
    \item For an $\mathbb{F}_2$-algebra $R$ let $\text{Aut}\mathbb{G}_M(R)$ be the groupoid whose objects are $\mathbb{F}_2$-algebra maps $M^*\xrightarrow{f}R$ and whose morphisms from $f$ to $g$ are strict isomorphisms of formal group laws $\phi$ from $f_*F_M$ to $g_*F_M$. Write $R_f$ for the induced $M^*$-algebra structure whose underlying ring is $R$. Note that $MR_f$ is canonically real oriented by the class $e_{MR_f}=[e_M\otimes 1]\in MR_f^*BC_2=M^*BC_2\otimes_{M^*}R_f$. Then there is a functor
    $$\gamma:\text{Aut}\mathbb{G}_M(R)\rightarrow MT$$
    which on objects sends $f$ to $MR_f$ and on morphisms sends $\phi$ to the transformation $\gamma(\phi):MR_f\rightarrow MR_g$ such that $\gamma(\phi)(e_{MR_f})=\phi(e_{MR_g})\in MR_g^*BC_2\simeq R_g[[e_{MR_g}]]$ as a power series.
\end{enumerate}
Then the evaluation map $ev_H$ from the previous lemma is an isomorphism. In particular if $\mathcal{A}_2$ denotes a Hopf algebra corepresenting $\text{Aut}\mathbb{G}_H$ then there is an isomorphism of Hopf algebras $H_*H\simeq\mathcal{A}_2.$
\end{lem}

\begin{proof}
By item 1 $\mathbb{F}_2$ is an algebra over $M^*$. By item 2 the formal group law of $M^*BC_2$ has vanishing 2-series $[2](e_M)=0$ and so it is isomorphic to the additive one. Transporting such an isomorphism $\phi$ with $\gamma$ produces a natural isomorphism
$$M^*X\simeq M^*X\otimes_{M^*}M^*\xrightarrow{\gamma(\phi)}M^*X\otimes_{M^*}\mathbb{F}_2\otimes_{\mathbb{F}_2}M^*=M\mathbb{F}_2\otimes_{\mathbb{F}_2}M^*. $$
Thus $M\mathbb{F}_2$ is a summand of a cohomology theory and hence a cohomology theory. The map in item 1 induces a map $M\mathbb{F}_2\rightarrow H$ which by Eilenberg-Steenrod uniqueness must be an isomorphism. Thus, combining item 3 with the pullback along the surjection $M\rightarrow M\mathbb{F}_2\simeq H$, the map $ev_H$ must be injective. So to produce an inverse it suffices to produce a section. Now since $\mathbb{F}_2$ is a field we have $MR\simeq HR$ for all $\mathbb{F}_2$-algebras $R$. Note that there is an inclusion $B\text{Aut}\mathbb{G}_H\hookrightarrow\text{Aut}\mathbb{F}_H$ since $F_M$ is sent to the additive formal group law by the map $M^*\rightarrow\mathbb{F}_2$ in item 1. Then the desired section of $ev_H$ is given as follows. Write $\mathbb{F}_2\xrightarrow{u}R$ for the unit map of an $\mathbb{F}_2$-algebra $R$. Start with an automorphism of $\text{Aut}\mathbb{G}_H(R)$ and view this as an automorphism of the object $M^*\rightarrow\mathbb{F}_2\rightarrow R$ in $\text{Aut}\mathbb{G}_M(R)$. Using $\gamma$ that produces an automorphism of $MR\simeq HR$, which one then precomposes with the morphism $H\rightarrow HR$ induced by $u$.
\end{proof}

\begin{lem} A cohomology theory $M$ satisfying the conditions of the previous lemma exists. It is $MO$.
\end{lem}

\begin{proof}
Let $e_{MO}$ denote the homotopy class of the inclusion $\mathbb{RP}^\infty\hookrightarrow\Sigma MO$. This exhibits $(MO,e_{MO})$ as the universal real oriented multiplicative cohomology theory. Hence there is a map $MO\rightarrow H$ sending $e_{MO}$ to $e\in H^*BC_2$. Thus items 1 and 2 are satisfied. Since $MO$ is real oriented, an easy calculation\footnote{The real orientation $e_{MO}$ implies that all the differentials in the relevant Atiyah-Hirzebruch sequences are trivial.} shows that $MO_*MO\simeq MO_*[a_1,a_2,...]$ which is free and hence flat over $MO^*$. Let $A$ be an $MO^*$-algebra and $T:MO\rightarrow MOA$ a multiplicative natural transformation. Let $\Theta_n\in MO^*MO(n))$ be the universal Thom class. By the splitting principle $T(\Theta_1)$ determines $T(\Theta_n)$, which determines all of $T$ by universality. Therefore if $MOA$ is a cohomology theory then $ev_{MO}(A)$ is the composition of an isomorphism ($\eta_{MO}(A)$) and an injection ($T\mapsto T(\Theta_1)$) so item 3 is satisfied.
Let $A$ be an $\mathbb{F}_2$-algebra receiving two maps $f,g:MO^*\rightarrow A$ and let $\phi\in A[[x]]$ be an isomorphism from $f_*F_{MO}$ to $g_*F_{MO}$. Note that $MOA_f$ and $MOA_g$ are canonically oriented by $e_{MOA_f}$ and $e_{MOA_g}$ as defined in item 4. It remains to construct (functorially) an automorphism $\gamma(\phi):MOA_f\rightarrow MOA_g$ such that $T(e_{MOA_f})=\phi(e_{MOA_g})$. By the calculation $MO_*MO\simeq MO_*\otimes_{\mathbb{F}_2}\mathbb{F}_2[a_1,a_2,...]$, the pair $(g,\phi)$ corresponds\footnote{By strictness $f(e_{MOA_g})=e_{MOA_g}+f_1e_{MOA_g}^2+...$ and the corresponding map $\Phi:MO_*[a_1,...]\rightarrow A$ sends $a_i$ to $f_{i+1}$.} to an $\mathbb{F}_2$-algebra map $\Phi:MO_*MO\rightarrow A$ and so using the natural transformation $\eta_{MO}$ from the preamble one gets multiplicative transformation $\eta_{MO}(\Phi):MO\rightarrow MOA_g$ with the property that $\eta_{MO}(\Phi)(e_{MO})=f(e_{MOA_g})$. Indeed the image of $e_{MO}$ under the map $MO^*BC_1\xrightarrow{1\wedge\text{id}}(MO\wedge MO)^*BC_2\simeq MO^*BC_2\otimes_{MO^*}MO_*MO$ is $[e_{MO}\otimes1] + [e_{MO}^2\otimes a_1]+...$ by a simple calculation using the pairing between homology and cohomology\footnote{See e.g. Adams [1] Lemma 6.3 page 60 for the complex version.}. Note that the map $\eta_{MO}(\Phi)(\text{pt}):MO^*\rightarrow A$ is not equal to $g$. Now $\eta_{MO}(\Phi)$ induces an $A$-linear map $MOA_f^*X\rightarrow MOA_g^*X$ and produces the desired map $\gamma(\phi):MOA_f\rightarrow MOA_g$ which is what we are really after. Functoriality\footnote{This argument for functoriality is due to Quillen in [4].} is proved by noting that $\gamma(\phi)$ is uniquely characterized by the properties of being $A$-linear, multiplicative, and sending $e_{MOA_f}$ to $\phi(e_{MOA_g})$.
\end{proof}
\begin{rem}
There is a similar but slightly messier derivation of the dual Steenrod algebra at odd primes which uses a mod $p$ analog of $MO$ constructed by Shaun Bullett [2] using manifolds with singularities. I decided to leave the odd primes to forthcoming work\footnote{This work will also discuss derivations of algebras of cohomology operations of other spectra as well as a derivation of the Dyer-Lashof algebra.} since the added technical difficulties (which in this business are usual at odd primes) obscures the simplicity and brevity of the exposition, which is a cardinal virtue of this note.
\end{rem}

\noindent
\textbf{Acknowledgements}: Thanks to Tim Campion for extensive comments on previous drafts, and to Robert Burklund for catching a gap in a proof.

\section{References}

[1] Adams, F. \textit{Stable homotopy and generalised homology}. Chicago Lectures in Mathematics. The University of Chicago Press, 1995.

\

\noindent
[2] Bullett, S. \textit{A Z/p Analogue for Unoriented Bordism}. University of Warwick PhD Thesis (1973)

\

\noindent
[3] Peterson, E. \textit{Formal geometry and bordism operations}. Cambridge Studies in Advanced Mathematics, v17 (2019)

\

\noindent
[4] Quillen, D. \textit{Elementary Proofs of Rome Results of Cobordism Theory Using Steenrod Operations}. Advances in Mathematics, 7 29-56 (1971)
\end{document}